\documentclass{article}
\usepackage[latin1]{inputenc}
\usepackage{amssymb, amsmath, amsthm}
\usepackage[a4paper]{geometry}

\usepackage{footnote}
\makesavenoteenv{tabular}
\usepackage{tabularx}
\usepackage{rotating}

\def\url#1{\expandafter\string\csname #1\endcsname}

\usepackage{enumitem}
\usepackage{color}
\usepackage{amsmath}

\theoremstyle{plain}

\newtheorem{corollary}{Corollary}
\newtheorem{definition}{Definition}[section]
\newtheorem{example}{Example}[section]
\newtheorem{lemma}{Lemma}[section]

\newtheorem{proposition}{Proposition}[section]
\newtheorem{remark}{Remark}
\newtheorem{theorem}{Theorem}[section]

\def\keywords{\vspace{.5em} 
{\textit{Keywords}:\,\relax%
}}

\usepackage{hyperref}
\begin{document}
\begin{center}
\Large{Ef{\kern0pt}f{\kern0pt}icient weight vectors from pairwise comparison matrices} \\
\end{center}

\begin{center}
S\'andor BOZ\'OKI$^{\,\,1,2,3}$
\footnotetext[1]{corresponding author}
\footnotetext[2]{Laboratory on Engineering and Management
Intelligence, Research Group of Operations Research and Decision
Systems, Institute for Computer Science and Control, Hungarian
Academy of Sciences (MTA SZTAKI); Mail: 1518 Budapest, P.O.~Box 63, Hungary.}
\footnotetext[3]{Department of Operations Research and Actuarial Sciences, Corvinus University
of Budapest, Hungary
 \textit{E-mail: bozoki.sandor@sztaki.mta.hu}}
J\'anos F\"UL\"OP$^{\,\,2,4}$
\footnotetext[4]{Institute of Applied Mathematics, John von Neumann Faculty of Informatics, \'Obuda University, Hungary
 \textit{E-mail: fulop.janos@sztaki.mta.hu}}

\end{center}
\begin{abstract}
Pairwise comparison matrices are frequently applied in multi-criteria decision making.
A weight vector is called ef{\kern0pt}f{\kern0pt}icient if no other weight vector is at least as good in approximating
the elements of the pairwise comparison matrix, and strictly better in at least one position.
A weight vector is weakly ef{\kern0pt}f{\kern0pt}icient if the pairwise ratios
cannot be improved in all non-diagonal positions.
We show that the principal eigenvector is always weakly ef{\kern0pt}f{\kern0pt}icient, but numerical examples show that
it can be inef{\kern0pt}f{\kern0pt}icient.
The linear programs proposed test whether a given weight vector is (weakly) ef{\kern0pt}f{\kern0pt}icient,
and in case of (strong) inef{\kern0pt}f{\kern0pt}iciency, an ef{\kern0pt}f{\kern0pt}icient (strongly) dominating
weight vector is calculated.
The proposed algorithms are implemented in Pairwise Comparison Matrix Calculator,  available at
pcmc.online.

\keywords{multiple criteria analysis, decision support, pairwise comparison matrix, Pareto
optimality, ef{\kern0pt}f{\kern0pt}iciency, linear programming}
\end{abstract}
\section{Introduction} \label{section:1} 
\label{intro}

\subsection{Pairwise comparison matrices}
Pairwise comparison matrix \cite{Saaty1977} has been a popular tool in multiple criteria decision making,
for weighting the criteria and evaluating the alternatives with respect to every criterion. Decision makers
compare two criteria or two alternatives at a time and judge which one is more important or better, and how many times.
Formally, a pairwise comparison matrix is a positive matrix $\mathbf{A}$ of size $n \times n$, where $n \geq 3$ denotes
the number of items to compare. Reciprocity is assumed: $a_{ij} = 1/a_{ji}$ for all $1 \leq i,j \leq n.$
A pairwise comparison matrix is called consistent, if $a_{ij} a_{jk} = a_{ik} $ for all $i,j,k.$
Let $PCM_n$ denote the set of pairwise comparison matrices of size $n \times n.$
Once the decision maker provides all the $n(n-1)/2$ comparisons, the objective is to f{\kern0pt}ind a weight
vector $\mathbf{w}=(w_1, w_2, \ldots, w_n)^{\top} \in \mathbb{R}^n$ such that the pairwise ratios of
the weights, $w_i/w_j$, are \emph{as close as possible} to the matrix elements $a_{ij}$.
Several methods have been suggested for this weighting problem,
e.g.,
the eigenvector method \cite{Saaty1977},
the least squares method \cite{Bozoki2008,ChuKalabaSpingarn1979,Fulop2008,Jensen1984},
the logarithmic least squares method \cite{CrawfordWilliams1980,CrawfordWilliams1985,deGraan1980},
the spanning tree approach
\cite{BozokiTsyganok2017,LundySirajGreco2017,OlenkoTsyganok2016,SirajMikhailovKeane2012a,SirajMikhailovKeane2012b,Tsyganok2000,Tsyganok2010}
besides many other proposals discussed and compared by
Golany and Kress \cite{GolanyKress1993},
Choo and Wedley \cite{ChooWedley2004},
Lin \cite{Lin2007},
Fedrizzi and Brunelli \cite{FedrizziBrunelli2010}.
Bajwa, Choo and Wedley \cite{BajwaChooWedley2008} not only compare seven weighting methods
with respect to four criteria, but provide a detailed list of nine earlier comparative studies, too.

\subsection{Weighting as a multiple objective optimization problem}

The weighting problem itself can be considered as
a multi-objective optimization problem which includes $n^2-n$ objective functions,
namely $\left| x_i/x_j - a_{ij}  \right|$, $1 \leq i \neq j \leq n.$
Let  $\mathbf{A} = \left[ a_{ij} \right]_{i,j=1,\ldots,n}$
be a pairwise comparison matrix and write the multi-objective optimization
problem
\begin{equation}
\min\limits_{\mathbf{x} \in \mathbb{R}^n_{++}}
\left| \frac{x_i}{x_j} - a_{ij}  \right|_{1 \leq i \neq j \leq n}. \label{eq:xMultiObjectiveOptimizationProblem} 
\end{equation}

Ef{\kern0pt}f{\kern0pt}iciency or Pareto optimality \cite[Chapter 2]{Miettinen1998} is
a key concept in multiple objective optimization and multiple criteria decision making.
See Ehrgott's historical overview \cite{Ehrgott2012}, beginning with Edgeworth \cite{Edgeworth1881} and
Pareto \cite{Pareto1906}.

Consider the functions
\begin{equation}
f_{ij}: \mathbb{R}^n_{++}\to \mathbb{R},\  i,j=1,\dots , n,
\nonumber
\end{equation}
def{\kern0pt}ined by
\begin{equation}
f_{ij}(\mathbf{x})= \left| \frac{x_i}{x_j} - a_{ij}  \right|,\
i,j=1,\dots , n, \label{eq:fijx}
\end{equation}
as in Blanquero, Carrizosa, Conde \cite[p.273]{BlanqueroCarrizosaConde2006}.
Since $f_{ii}(\mathbf{x})=0$ for every $\mathbf{x} \in
\mathbb{R}^n_{++}$ and $i=1,\dots , n$, these constant functions
are irrelevant from the aspect of multi-objective optimization, so
they will be simply left out from the investigations.

Let the vector-valued function $\mathbf{f}: \mathbb{R}^n_{++}\to
\mathbb{R}^{n(n-1)}_{++}$ def{\kern0pt}ined by its components $f_{ij},
i,j=1,\dots , n,\, i\ne j$. Consider the problem of minimizing
$\mathbf{f}$ over a nonempty set $X\subseteq \mathbb{R}^n$ that
can be written in the general form of the
 \textit{vector optimization problem}
\begin{equation}
\min\limits_{\mathbf{x} \in X} \mathbf{f}(\mathbf{x}).
\label{eq:VOP}
\end{equation}
With $X=\mathbb{R}^n_{++}$, where the latter denotes the positive
orthant in $\mathbb{R}^n$, we get problem
(\ref{eq:xMultiObjectiveOptimizationProblem}) in a bit more
general form.

Recall the following basic concepts used for multiple objective or
vector optimization. A point $\mathbf{\bar x}\in X$ is said to be
an ef{\kern0pt}f{\kern0pt}icient solution of (\ref{eq:VOP}) if there is no
$\mathbf{x}\in X$ such that $\mathbf{f} (\mathbf{x})\le \mathbf{f}
(\mathbf{\bar x})$, $\mathbf{f}(\mathbf{x})\ne
\mathbf{f}(\mathbf{\bar x})$, meaning that $f_{ij}(\mathbf{x})\le
f_{ij}(\mathbf{\bar x})$ for all $i\ne j$ with strict inequality
for at least one index pair $i\ne j$. In the literature, the names
Pareto-optimal, nondominated and noninferior solution are also
used instead of ef{\kern0pt}f{\kern0pt}icient solution.

A point $\mathbf{\bar x}\in X$ is said to be a weakly ef{\kern0pt}f{\kern0pt}icient
solution of (\ref{eq:VOP}) if there is no $\mathbf{x}\in X$ such
that $\mathbf{f} (\mathbf{x})< \mathbf{f} (\mathbf{\bar x})$,
i.e.\ $f_{ij}(\mathbf{x})< f_{ij}(\mathbf{\bar x})$ for all $i\ne
j$. Ef{\kern0pt}f{\kern0pt}icient solutions are sometimes called strongly ef{\kern0pt}f{\kern0pt}icient.

A point $\mathbf{\bar x}\in X$ is said to be a  locally ef{\kern0pt}f{\kern0pt}icient
solution of (\ref{eq:VOP}) if there exists $\delta >0$ such that
$\mathbf{\bar x}$ is an ef{\kern0pt}f{\kern0pt}icient  solution in $X\cap
B(\mathbf{\bar x},\delta)$, where $B(\mathbf{\bar x},\delta)$ is a
$\delta$-neighborhood around $\mathbf{\bar x}$. The local weak
ef{\kern0pt}f{\kern0pt}iciency is def{\kern0pt}ined similarly for a point $\mathbf{\bar x}\in
X$, the only dif{\kern0pt}ference is that weakly ef{\kern0pt}f{\kern0pt}icient solutions are
considered instead of ef{\kern0pt}f{\kern0pt}icient solutions. \\

Several multi-objective optimization models have been proposed in the research of pairwise comparison matrices.
Departing from \cite{Mikhailov2006}, Mikhailov and Knowles \cite{MikhailovKnowles2009} include two objective functions,
the sum of least squares, written for the upper diagonal positions, and the number of minimum violations,
then apply an evolutionary algorithm to generate the Pareto frontier.
A third objective function, the total deviation from second-order indirect judgments, is added
in \cite{SirajMikhailovKeane2012c}.

The $n(n-1)/2$ objective functions $\left| x_i/x_j - a_{ij}  \right|$, $1 \leq i \neq j \leq n,$ of
the multi-objective optimization problem (\ref{eq:xMultiObjectiveOptimizationProblem})
can be aggregated into a single objective function in several ways. Their sum gives the weighting method
\emph{least absolute error} \cite[Section 4, LAE]{ChooWedley2004}.
If their maximum is taken into consideration, weighting method
\emph{least worst absolute error} \cite[Section 4, LWAE]{ChooWedley2004} is resulted in.
The sum of their squares is the classical \emph{least squares method} \cite{Bozoki2008,ChuKalabaSpingarn1979,Fulop2008,Jensen1984}.
A (parametric) linear combination of the sum and the maximum is proposed by Jones and Mardle \cite{JonesMardle2004}
to f{\kern0pt}ind a compromise weight vector.
A similar idea is applied in the proposal of Dopazo and Ruiz-Tagle \cite{DopazoRuiz-Tagle2011},
developed for group decision problems with incomplete pairwise comparison matrices.

In the rest of the paper ef{\kern0pt}f{\kern0pt}iciency for problem  (\ref{eq:xMultiObjectiveOptimizationProblem}),
including $n(n-1)/2$ objective functions, is considered.
The explicit presentation will be unavoidable for the problem specif{\kern0pt}ic
concept of internal ef{\kern0pt}f{\kern0pt}iciency introduced recently in
\cite{Bozoki2014}.

\subsection{Ef{\kern0pt}f{\kern0pt}iciency of weight vectors}

Let
$\mathbf{w} = (w_1, w_2, \ldots, w_n)^{\top}$ be a positive weight vector.

\begin{definition} \label{def:DefinitionEfficient}  
Weight vector $\mathbf{w}$ is called \emph{ef{\kern0pt}f{\kern0pt}icient} for  (\ref{eq:xMultiObjectiveOptimizationProblem})
if no positive weight vector
$\mathbf{w^{\prime}} = (w^{\prime}_1, w^{\prime}_2, \ldots, w^{\prime}_n)^{\top}$
exists such that
\begin{align}
 \left|a_{ij} - \frac{w^{\prime}_i}{w^{\prime}_j} \right| &\leq \left|a_{ij} - \frac{w_i}{w_j} \right| \qquad \text{ for all } 1 \leq i,j \leq n, \label{def:DefinitionEfficientProperty1}  \\   
 \left|a_{k{\ell}} - \frac{w^{\prime}_k}{w^{\prime}_{\ell}} \right| &<  \left|a_{k{\ell}} - \frac{w_k}{w_{\ell}} \right|  \qquad \text{ for some } 1 \leq k,\ell \leq n.  \label{def:DefinitionEfficientProperty2}    
\end{align}
\end{definition}
Weight vector $\mathbf{w}$ is called \emph{inef{\kern0pt}f{\kern0pt}icient} for  (\ref{eq:xMultiObjectiveOptimizationProblem}) if it is not ef{\kern0pt}f{\kern0pt}icient for  (\ref{eq:xMultiObjectiveOptimizationProblem}). \\

If weight vector $\mathbf{w}$ is inef{\kern0pt}f{\kern0pt}icient for  (\ref{eq:xMultiObjectiveOptimizationProblem}) and
weight vector $\mathbf{w^{\prime}}$ fulf{\kern0pt}ils (\ref{def:DefinitionEfficientProperty1})-(\ref{def:DefinitionEfficientProperty2}),
we say that $\mathbf{w^{\prime}}$ \emph{dominates} $\mathbf{w}$. Note that dominance is transitive.

It follows from the def{\kern0pt}inition that an arbitrary rescaling
does not inf{\kern0pt}luence (in)ef{\kern0pt}f{\kern0pt}iciency.
\begin{remark}  \label{remark:arbitrarynormalization}
A weight vector $\mathbf{w}$ is ef{\kern0pt}f{\kern0pt}icient for  (\ref{eq:xMultiObjectiveOptimizationProblem})
if and only if $c\mathbf{w}$ is ef{\kern0pt}f{\kern0pt}icient for  (\ref{eq:xMultiObjectiveOptimizationProblem}),
where $c >0$ is an arbitrary scalar.
\end{remark}

\begin{example}  \label{example:0a}
Consider four criteria $C_1, C_2, C_3, C_4,$ pairwise comparison matrix $\mathbf{A} \in PCM_4$ and its principal right
eigenvector $\mathbf{w}$ as follows:
\begin{equation*}
\mathbf{A} =
\begin{pmatrix}
$\,\,$   1  $\,\,$   & $\,\,$    1  $\,\,$   & $\,\,$    4   $\,\,$ & $\,\,$    9   $\,\,$     \\
$\,\,$   1  $\,\,$   & $\,\,$    1  $\,\,$   & $\,\,$    7   $\,\,$ & $\,\,$    5   $\,\,$     \\
$\,\,$  1/4 $\,\,$   & $\,\,$   1/7 $\,\,$   & $\,\,$    1   $\,\,$ & $\,\,$    4   $\,\,$     \\
$\,\,$  1/9 $\,\,$   & $\,\,$   1/5 $\,\,$   & $\,\,$   1/4  $\,\,$ & $\,\,$    1   $\,\,$
\end{pmatrix},
\qquad \qquad {\mathbf{w}} =
\begin{pmatrix}
0.404518  \\
0.436173 \\
0.110295  \\
0.049014
\end{pmatrix},
\qquad \qquad {\mathbf{w}^{\prime}} =
\begin{pmatrix}
\mathbf{0.441126}  \\
0.436173 \\
0.110295  \\
0.049014
\end{pmatrix}. 
\end{equation*}
In order to prove the inef{\kern0pt}f{\kern0pt}iciency of the
principal right eigenvector $\mathbf{w}$, let us increase its f{\kern0pt}irst
coordinate: ${w}^{\prime}_1 :=  9{w}_4  = 0.441126, \,
{w}^{\prime}_i := {w}_i, i =2,3,4.$ The consistent approximations generated
by weight vectors ${\mathbf{w}}, {\mathbf{w}}^{\prime},$
\begin{equation}
\left[ \frac{{w}_i}{{w}_j} \right] =
\begin{pmatrix}
$\,$   1    $\,\,$ & $\,\,$ 0.9274 $\,\,$ & $\,\,$ 3.6676 $\,\,$ & $\,\,$ 8.2531 $\,$    \\
$\,$ 1.0783 $\,\,$ & $\,\,$   1    $\,\,$ & $\,\,$ 3.9546 $\,\,$ & $\,\,$ 8.8989 $\,$    \\
$\,$ 0.2727 $\,\,$ & $\,\,$ 0.2529 $\,\,$ & $\,\,$   1    $\,\,$ & $\,\,$ 2.2503 $\,$   \\
$\,$ 0.1212 $\,\,$ & $\,\,$ 0.1124 $\,\,$ & $\,\,$ 0.4444 $\,\,$ & $\,\,$   1    $\,$
\end{pmatrix},   \label{example:0awiperwj}    
\end{equation}
\[
\left[ \frac{{w}^{\prime}_i}{{w}^{\prime}_j} \right] =
\begin{pmatrix}
$\,$   1    $\,\,$ & $\,\,$   \mathbf{1.0114}    $\,\,$ & $\,\,$ \mathbf{3.9995} $\,\,$ & $\,\,$ \mathbf{9} $\,$    \\
$\,$   \mathbf{0.9888}    $\,\,$ & $\,\,$   1    $\,\,$ & $\,\,$ 3.9546 $\,\,$ & $\,\,$ 8.8989 $\,$    \\
$\,$ \mathbf{0.2500} $\,\,$ & $\,\,$ 0.2529 $\,\,$ & $\,\,$   1    $\,\,$ & $\,\,$ 2.2503 $\,$   \\
$\,$ \mathbf{0.1111} $\,\,$ & $\,\,$ 0.1124 $\,\,$ & $\,\,$ 0.4444 $\,\,$ & $\,\,$   1    $\,$
\end{pmatrix},
        \]
show that inequality (\ref{def:DefinitionEfficientProperty1})
holds for all $1\leq i,j \leq 4$, and the strict inequality
(\ref{def:DefinitionEfficientProperty2}) holds for $(k,\ell) \in
\{(1,2), (1,3), (1,4), (2,1), (3,1), (4,1)  \}.$
For example, with $k=1, \ell=2, |\frac{{w}^{\prime}_1}{{w}^{\prime}_2} - a_{12}|
= |1.0114 - 1 | =0.0114 < |\frac{{w}_1}{{w}_2} - a_{12}| = |0.9274 - 1| = 0.0726.$
Weight vector ${\mathbf{w}}^{\prime}$ dominates ${\mathbf{w}}$.
Note that the principal right eigenvector $\mathbf{w}$ ranks the criteria as
$C_2 \succ C_1 \succ C_3 \succ C_4,$ while the dominating weight
vector $\mathbf{w}^{\prime}$ ranks them as  $C_1 \succ C_2 \succ C_3 \succ C_4.$
\end{example}

Blanquero et al. (2006) considered the local variant of ef{\kern0pt}f{\kern0pt}iciency:
\begin{definition} \label{def:DefinitionlocalEfficient} 
Weight vector $\mathbf{w}$ is called \emph{locally ef{\kern0pt}f{\kern0pt}icient}
 for (\ref{eq:xMultiObjectiveOptimizationProblem})
if there exists a neighborhood of $\mathbf{w}$, denoted by $V(\mathbf{w}),$
such that no positive weight vector $\mathbf{w^{\prime}} \in V(\mathbf{w}) $
fulf{\kern0pt}illing (\ref{def:DefinitionEfficientProperty1})-(\ref{def:DefinitionEfficientProperty2}) exists.
\end{definition}
Weight vector $\mathbf{w}$ is called \emph{locally inef{\kern0pt}f{\kern0pt}icient} if it is
not locally ef{\kern0pt}f{\kern0pt}icient. \\

Another variant of (in)ef{\kern0pt}f{\kern0pt}iciency has been introduced
by Boz\'oki (2014):
\begin{definition} \label{def:DefinitionInternallyEfficient} 
Weight vector $\mathbf{w}$ is called \emph{internally ef{\kern0pt}f{\kern0pt}icient}
 for (\ref{eq:xMultiObjectiveOptimizationProblem})
if no positive weight vector \linebreak
$\mathbf{w^{\prime}} = (w^{\prime}_1, w^{\prime}_2, \ldots, w^{\prime}_n)^{\top}$
exists such that
\begin{align}
\left.
\begin{array}{ccc}
a_{ij} \leq \frac{w_i}{w_j} & \Longrightarrow &
a_{ij} \leq \frac{w^{\prime}_i}{w^{\prime}_j} \leq \frac{w_i}{w_j} \\
a_{ij} \geq \frac{w_i}{w_j} & \Longrightarrow &
a_{ij} \geq \frac{w^{\prime}_i}{w^{\prime}_j} \geq \frac{w_i}{w_j}
\end{array}
\right\}
\qquad
\text{ for all } 1 \leq i,j \leq n, \label{def:DefinitionInternallyEfficientProperty1}  \\   
\left.
\begin{array}{ccc}
a_{k{\ell}} \leq \frac{w_k}{w_{\ell}} & \Longrightarrow &
\frac{w^{\prime}_k}{w^{\prime}_{\ell}} < \frac{w_k}{w_{\ell}} \\
a_{k{\ell}} \geq \frac{w_k}{w_{\ell}} & \Longrightarrow &
\frac{w^{\prime}_k}{w^{\prime}_{\ell}} > \frac{w_k}{w_{\ell}}
\end{array}
\right\}
\qquad
\text{ for some } 1 \leq k,\ell \leq n.  \label{def:DefinitionInternallyEfficientProperty2}    
\end{align}
\end{definition}
Weight vector $\mathbf{w}$ is called \emph{internally inef{\kern0pt}f{\kern0pt}icient}
if it is not internally ef{\kern0pt}f{\kern0pt}icient. \\

If weight vector $\mathbf{w}$ is inef{\kern0pt}f{\kern0pt}icient for  (\ref{eq:xMultiObjectiveOptimizationProblem}) and
weight vector $\mathbf{w^{\prime}}$ fulf{\kern0pt}ils (\ref{def:DefinitionInternallyEfficientProperty1})-(\ref{def:DefinitionInternallyEfficientProperty2}),
we say that $\mathbf{w^{\prime}}$ \emph{dominates} $\mathbf{w}$ \emph{internally}.
Note that internal dominance is transitive.

\begin{example}  \label{example:0b} 
Consider the pairwise comparison matrix $\mathbf{A} \in PCM_4$ of Example \ref{example:0a}
and its principal right eigenvector $\mathbf{w}$.
Now let us increase the f{\kern0pt}irst coordinate of $\mathbf{w}$ until it reaches the second one,
\[
\qquad \qquad {\mathbf{w}^{\prime\prime}} =
\begin{pmatrix}
\mathbf{0.436173}  \\
0.436173 \\
0.110295  \\
0.049014
\end{pmatrix}.
\]
The consistent approximation generated by weight vector ${\mathbf{w}}^{\prime\prime}$ is as follows:
\begin{equation} \label{example:0bwiperwj}  
\left[ \frac{{w}^{\prime\prime}_i}{{w}^{\prime\prime}_j} \right] =
\begin{pmatrix}
$\,$   1    $\,\,$ & $\,\,$   \mathbf{1}    $\,\,$ & $\,\,$ \mathbf{3.9546} $\,\,$ & $\,\,$ \mathbf{8.8989} $\,$    \\
$\,$ \mathbf{1}    $\,\,$ & $\,\,$   1    $\,\,$ & $\,\,$ 3.9546 $\,\,$ & $\,\,$ 8.8989 $\,$    \\
$\,$ \mathbf{0.2529} $\,\,$ & $\,\,$ 0.2529 $\,\,$ & $\,\,$   1    $\,\,$ & $\,\,$ 2.2503 $\,$   \\
$\,$ \mathbf{0.1124} $\,\,$ & $\,\,$ 0.1124 $\,\,$ & $\,\,$ 0.4444 $\,\,$ & $\,\,$   1    $\,$
\end{pmatrix}.
\end{equation}
Inequality (\ref{def:DefinitionInternallyEfficientProperty1})
holds for all $1 \leq i,j \leq 4$, and the strict inequality
(\ref{def:DefinitionInternallyEfficientProperty2}) holds for
\linebreak $(k,\ell) \in \{(1,2), (1,3), (1,4), (2,1), (3,1),
(4,1)  \}.$ Weight vector ${\mathbf{w}}^{\prime\prime}$  dominates
${\mathbf{w}}$ internally. Observe that weight vector
${\mathbf{w}}^{\prime}$ in Example \ref{example:0a} does not dominate
${\mathbf{w}}$ internally. Note that the internally dominating
weight vector $\mathbf{w}^{\prime\prime}$
ranks the criteria as  $C_1 \sim C_2 \succ C_3 \succ C_4.$ \\

The local inef{\kern0pt}f{\kern0pt}iciency of weight vector ${\mathbf{w}}$ can be checked by the
fact that weight vector $(w_1+\varepsilon,w_2,w_3,w_4)^{\top}$ dominates ${\mathbf{w}}$ for all
$\varepsilon < 2(w_2 - w_1) = 0.0633$, furthermore it dominates ${\mathbf{w}}$ internally for all
$\varepsilon < w_2 - w_1 = 0.0316$, providing the same ranking $C_2 \succ C_1 \succ C_3 \succ C_4$
as of the principal right eigenvector ${\mathbf{w}}$.
\end{example}

A natural question might arise. How can dominating weight vectors in Examples \ref{example:0a}--\ref{example:0b}
be found? We premise that algorithmic ways of f{\kern0pt}inding a dominating ef{\kern0pt}f{\kern0pt}icient weight vector
shall be given in details in Section \ref{section:4}.

It follows from the def{\kern0pt}initions that if
weight vector $\mathbf{w}$ is internally inef{\kern0pt}f{\kern0pt}icient, then
it is inef{\kern0pt}f{\kern0pt}icient.
Blanquero, Carrizosa and Conde proved that the two def{\kern0pt}initions are in fact equivalent:

\begin{theorem} \label{theorem:EfficientlocalEfficient}
\cite[Theorem 3]{BlanqueroCarrizosaConde2006}
Weight vector $\mathbf{w}$ is ef{\kern0pt}f{\kern0pt}icient
 for (\ref{eq:xMultiObjectiveOptimizationProblem})
if and only if it is locally ef{\kern0pt}f{\kern0pt}icient
 for (\ref{eq:xMultiObjectiveOptimizationProblem}),
i.e., Def{\kern0pt}initions \ref{def:DefinitionEfficient} and \ref{def:DefinitionlocalEfficient}
are equivalent.
\end{theorem}

\begin{proposition} \label{proposition:EfficientInternallyEfficient}
Weight vector $\mathbf{w}$ is ef{\kern0pt}f{\kern0pt}icient
 for (\ref{eq:xMultiObjectiveOptimizationProblem})
 if and only if it is internally ef{\kern0pt}f{\kern0pt}icient
 for (\ref{eq:xMultiObjectiveOptimizationProblem}),
i.e., Def{\kern0pt}initions \ref{def:DefinitionEfficient} and \ref{def:DefinitionInternallyEfficient}
are equivalent.
\end{proposition}

\begin{proof}
Suf{\kern0pt}f{\kern0pt}iciency follows by def{\kern0pt}inition. For necessity, it is more
convenient to show that inef{\kern0pt}f{\kern0pt}iciency implies internal inef{\kern0pt}f{\kern0pt}iciency.
Let weight vector $\mathbf{w}$ be inef{\kern0pt}f{\kern0pt}icient.
Theorem \ref{theorem:EfficientlocalEfficient} implies that
$\mathbf{w}$ is locally inef{\kern0pt}f{\kern0pt}icient as well, i.e.,
there exists $\mathbf{w^{\prime}}$ in any neighborhood $U(\mathbf{w})$
such that $\mathbf{w^{\prime}}$ dominates $\mathbf{w}.$
If $U(\mathbf{w})$ is suf{\kern0pt}f{\kern0pt}iciently small, then
\begin{align} \label{eq:wprimedominatesw} 
\begin{array}{ccc}
a_{ij} < \frac{w_i}{w_j} & \Longrightarrow &
a_{ij} < \frac{w^{\prime}_i}{w^{\prime}_j} \leq \frac{w_i}{w_j} \\[2mm]
a_{ij} > \frac{w_i}{w_j} & \Longrightarrow &
a_{ij} > \frac{w^{\prime}_i}{w^{\prime}_j} \geq \frac{w_i}{w_j} \\[2mm]
a_{ij} = \frac{w_i}{w_j} & \Longrightarrow &
a_{ij} = \frac{w^{\prime}_i}{w^{\prime}_j} = \frac{w_i}{w_j},
\end{array}
\end{align}
implying that $\mathbf{w}$ is internally inef{\kern0pt}f{\kern0pt}icient.
\end{proof}

\begin{corollary} \label{corollary:EfficiencyDefinitionsAreEquivalent}
Ef{\kern0pt}f{\kern0pt}iciency (Def{\kern0pt}inition \ref{def:DefinitionEfficient}),
local ef{\kern0pt}f{\kern0pt}iciency (Def{\kern0pt}inition \ref{def:DefinitionlocalEfficient}) and
internal ef{\kern0pt}f{\kern0pt}iciency (Def{\kern0pt}inition \ref{def:DefinitionInternallyEfficient}) are equivalent.
\end{corollary}

\begin{definition} \label{def:DirectedGraphEfficient}  
Let $\mathbf{A} =
\left[
a_{ij}
\right]_{i,j=1,\ldots,n} \in PCM_n$  and
$\mathbf{w} = (w_1, w_2, \ldots, w_n)^{T}$ be a positive weight vector.
 A directed graph $G=(V,\overrightarrow{E})_{\mathbf{A},\mathbf{w}}$ is def{\kern0pt}ined as follows:
$V=\{1,2,\ldots,n\} $ and
\[
 \overrightarrow{E} = \left\{ \text{arc(}  i \rightarrow j\text{)}   \left|  \frac{w_i}{w_j} \geq a_{ij}, i \neq j \right. \right\}.
\]
\end{definition}
It follows from Def{\kern0pt}inition \ref{def:DirectedGraphEfficient} that if $w_i/w_j=a_{ij}$, then there is a bidirected arc between nodes $i,j$.
The fundamental theorem of Blanquero, Carrizosa and Conde using the directed graph representation above is as follows:
\begin{theorem}[{\cite[Corollary 10]{BlanqueroCarrizosaConde2006}}] \label{thm:TheoremDirectedGraphEfficient}
Let $\mathbf{A} \in {PCM}_n$.
A weight vector $\mathbf{w}$ is ef{\kern0pt}f{\kern0pt}icient
 for (\ref{eq:xMultiObjectiveOptimizationProblem})
if and only if
$G=(V,\overrightarrow{E})_{\mathbf{A},\mathbf{w}}$
is  a strongly connected digraph, that is,
there exist directed paths from $i$ to $j$ and from $j$ to $i$ for all
pairs of  nodes  $i , j$.
\end{theorem}

Blanquero, Carrizosa and Conde {\cite[Remark 12]{BlanqueroCarrizosaConde2006}}
and Conde and P\'erez \cite[Theorem 2.2]{CondePerez2010} consider weak ef{\kern0pt}f{\kern0pt}iciency as follows:
\begin{definition} \label{def:DefinitionWeaklyEfficient}  
Weight vector $\mathbf{w}$ is called \emph{weakly
ef{\kern0pt}f{\kern0pt}icient}
 for (\ref{eq:xMultiObjectiveOptimizationProblem})
 if no positive weight vector
$\mathbf{w^{\prime}} = (w^{\prime}_1, w^{\prime}_2, \ldots,
w^{\prime}_n)^{\top}$ exists such that
\begin{align}
 \left|a_{ij} - \frac{w^{\prime}_i}{w^{\prime}_j} \right| &< \left|a_{ij} - \frac{w_i}{w_j} \right| \qquad \text{ for all } 1 \leq i \neq j \leq n, \label{def:WeakEfficiency}    
\end{align}
and weight vector $\mathbf{w}$ is called \emph{strongly inef{\kern0pt}f{\kern0pt}icient} if it is not weakly ef{\kern0pt}f{\kern0pt}icient. \\
\end{definition}

If weight vector $\mathbf{w}$ is strongly inef{\kern0pt}f{\kern0pt}icient
for  (\ref{eq:xMultiObjectiveOptimizationProblem}) and weight vector
$\mathbf{w^{\prime}}$ fulf{\kern0pt}ils (\ref{def:WeakEfficiency}),
we say that $\mathbf{w^{\prime}}$ \emph{strongly dominates} $\mathbf{w}$.
Note that strong dominance is transitive.

\begin{example}  \label{example:1}
Let $n \geq 3$ integer and $c, d > 0, c \neq d$ arbitrary. Let
$\mathbf{A} \in PCM_n$ be a consistent pairwise comparison matrix def{\kern0pt}ined as
$a_{ij} = c^{j-i}, \, \, i,j=1,\ldots,n$. Let weight vector
$\mathbf{w}$ be def{\kern0pt}ined by $w_i = d^{n+1-i}, i=1,\ldots,n$. Then
weight vector $\mathbf{w^{\prime}}$, def{\kern0pt}ined by $w^{\prime}_i =
c^{n+1-i}, i=1,\ldots,n$ provides strictly better approximation to
all non-diagonal elements of $\mathbf{A}$ than $\mathbf{w}$ does,
therefore $\mathbf{w}$ is strongly
inef{\kern0pt}f{\kern0pt}icient. The example is specif{\kern0pt}ied for
$n=4, c=2, d=3$ below.
\begin{align*}
\mathbf{A} =  \left[\frac{w^{\prime}_i}{w^{\prime}_j}\right]_{i,j=1,\ldots,4} &=
\begin{pmatrix}
$\,\,$   1   $\,\,$ & $\,\,$   2   $\,\,$ & $\,\,$   4   $\,\,$ & $\,\,$   8   $\,\,$   \\
$\,\,$  1/2  $\,\,$ & $\,\,$   1   $\,\,$ & $\,\,$   2   $\,\,$ & $\,\,$   4   $\,\,$   \\
$\,\,$  1/4  $\,\,$ & $\,\,$  1/2  $\,\,$ & $\,\,$   1   $\,\,$ & $\,\,$   2   $\,\,$   \\
$\,\,$  1/8  $\,\,$ & $\,\,$  1/4  $\,\,$ & $\,\,$  1/2  $\,\,$ & $\,\,$   1   $\,\,$
\end{pmatrix},
\qquad
\qquad
\mathbf{w^{\prime}} =
\begin{pmatrix}
8  \\
4  \\
2  \\
1
\end{pmatrix},    \\
\left[\frac{w_i}{w_j}\right]_{i,j=1,\ldots,4} &=
\begin{pmatrix}
$\,\,$   1   $\,\,$ & $\,\,$   3   $\,\,$ & $\,\,$   9   $\,\,$ & $\,\,$  27   $\,\,$   \\
$\,\,$  1/3  $\,\,$ & $\,\,$   1   $\,\,$ & $\,\,$   3   $\,\,$ & $\,\,$   9   $\,\,$   \\
$\,\,$  1/9  $\,\,$ & $\,\,$  1/3  $\,\,$ & $\,\,$   1   $\,\,$ & $\,\,$   3   $\,\,$   \\
$\,\,$  1/27 $\,\,$ & $\,\,$  1/9  $\,\,$ & $\,\,$  1/3  $\,\,$ & $\,\,$   1   $\,\,$
\end{pmatrix},
\qquad
\quad
\mathbf{w} =
\begin{pmatrix}
27  \\
9   \\
3   \\
1
\end{pmatrix}.
\end{align*}
\end{example}

\subsection{Ef{\kern0pt}f{\kern0pt}iciency and distance minimization}

Distance minimization does not necessarily induce ef{\kern0pt}f{\kern0pt}iciency.
Blanquero, Carrizosa and Conde \cite{BlanqueroCarrizosaConde2006}
and Fedrizzi \cite{Fedrizzi2013} showed that if the metric is
componentwise strictly increasing, then ef{\kern0pt}f{\kern0pt}iciency is implied.
\begin{definition} (\cite{Fedrizzi2013})
A metric $D: PCM_n \times PCM_n \rightarrow \mathbb{R}$ is called \emph{strictly monotonic}, if
$\left| a_{ij} - \frac{x_i}{x_j} \right| \leq \left| a_{ij} - \frac{y_i}{y_j} \right|$
for all $(i,j)$ and the inequality is strict for at least one pair of indices $(i,j)$, imply that
$ D(\mathbf{A}, \left[ \frac{x_i}{x_j} \right]) < D(\mathbf{A}, \left[ \frac{y_i}{y_j} \right])$.
\end{definition}
\begin{theorem} \label{thm:StrictlyMonotonicDistance} (\cite[Section 2]{BlanqueroCarrizosaConde2006},\cite{Fedrizzi2013})
A weight vector induced by a strictly monotonic metric is ef{\kern0pt}f{\kern0pt}icient
 for (\ref{eq:xMultiObjectiveOptimizationProblem}).
\end{theorem}
Theorem \ref{thm:StrictlyMonotonicDistance}
implies that the least squares method \cite{Bozoki2008,ChuKalabaSpingarn1979,Fulop2008,Jensen1984} with the objective function
$\sum\limits_{i,j} \left| a_{ij} - \frac{w_i}{w_j} \right|^2$ induces ef{\kern0pt}f{\kern0pt}icient weight vector(s).
Furthermore, power 2 can be replaced by an arbitrary $p \geq 1$, ef{\kern0pt}f{\kern0pt}iciency is kept.\\

Blanquero, Carrizosa and Conde \cite[Corollary 7]{BlanqueroCarrizosaConde2006} proved that the
logarithmic least squares method \cite{CrawfordWilliams1980,CrawfordWilliams1985,deGraan1980,LundySirajGreco2017}
with the objective function $\sum\limits_{i,j} \left( \log a_{ij} - \log \frac{w_i}{w_j} \right)^2$
yields an ef{\kern0pt}f{\kern0pt}icient solution (the row geometric mean).\\

The eigenvector method \cite{Saaty1977} is special, because we have seen in Example \ref{example:0a}
that the principal right eigenvector can be inef{\kern0pt}f{\kern0pt}icient. On the other hand, Fichtner
\cite{Fichtner1984,Fichtner1986} showed that the eigenvector method can be written
as a distance minimizing problem. Note that Fichtner's metric is neither continuous, nor strictly monotonic.\\

\subsection{Results of the paper}
The rest of the paper is organized as follows.
Section \ref{section:2} investigates that the formally dif{\kern0pt}ferent
def{\kern0pt}initions of (weak) ef{\kern0pt}f{\kern0pt}iciency are in fact
equivalent. It is shown that the set of (strongly) dominating weight vectors is convex.
Weak ef{\kern0pt}f{\kern0pt}iciency of the principal eigenvector is proved in Section \ref{section:3}.
A linear program is developed in Section \ref{section:4} in order to test whether a given weight
vector, with respect to a f{\kern0pt}ixed pairwise comparison matrix, is ef{\kern0pt}f{\kern0pt}icient.
If it is inef{\kern0pt}f{\kern0pt}icient, an ef{\kern0pt}f{\kern0pt}icient dominating weight vector is found.
Another linear program is constructed in Section \ref{section:5} for testing weak ef{\kern0pt}f{\kern0pt}iciency.
Again, if the weight vector is found to be strongly inef{\kern0pt}f{\kern0pt}icient, a strongly dominating weight vector is
calculated. Linear programs are implemented in Pairwise Comparison Matrix Calculator, available at pcmc.online.
Section \ref{section:6} concludes and raises some open questions.

\section{Equivalent def{\kern0pt}initions of ef{\kern0pt}f{\kern0pt}iciency and weak ef{\kern0pt}f{\kern0pt}iciency}
\label{section:2} 

In line with the ef{\kern0pt}f{\kern0pt}icient case,
locally and internally weakly ef{\kern0pt}f{\kern0pt}icient points can also be def{\kern0pt}ined in an explicit, problem-specif{\kern0pt}ic form.

Let $E$, $E_L$ and $E_I$
denote the set of the ef{\kern0pt}f{\kern0pt}icient, locally ef{\kern0pt}f{\kern0pt}icient and internally ef{\kern0pt}f{\kern0pt}icient solutions, respectively.
Similarly, let $W\!E$, $W\!E_L$ and $W\!E_I$
denote the set of the
weakly
ef{\kern0pt}f{\kern0pt}icient, locally
weakly
ef{\kern0pt}f{\kern0pt}icient and internally
weakly
ef{\kern0pt}f{\kern0pt}icient solutions, respectively.

According to Def{\kern0pt}inition  (\ref{def:DefinitionWeaklyEfficient}),
$$  W\!E = \{\mathbf{w}>\mathbf{0}\mid\text{there exists no }\mathbf{w^{\prime}}>\mathbf{0} \text{ for which (\ref{def:WeakEfficiency}) holds}\}.
$$

In the same way,
\begin{align}
  W\!E_L  =\{\mathbf{w}>\mathbf{0}\mid  &\text{ there exists a neighbourhood }
U(\mathbf{w}) \text{ such that} \nonumber \\
  &\text{ there exists no }  \mathbf{w^{\prime}} \in U(\mathbf{w}) \text{ for which }
(\ref{def:WeakEfficiency})
\text{ holds}\} \nonumber
\end{align}
and
\begin{align}
W\!E_I = \{\mathbf{w}>\mathbf{0}\mid  &
 \text{ there exists no }\mathbf{w^{\prime}}>\mathbf{0}
\text{ such that}\nonumber \\
&\begin{array}{lll}
a_{ij} \leq \frac{w_i}{w_j} & \Longrightarrow &
a_{ij} \leq \frac{w^{\prime}_i}{w^{\prime}_j} <\frac{w_i}{w_j}
\quad
\text{ for all } 1 \leq i\ne j \leq n,\\
a_{ij} \ge \frac{w_i}{w_j} & \Longrightarrow &
a_{ij} \ge \frac{w^{\prime}_i}{w^{\prime}_j} > \frac{w_i}{w_j}
\quad
\text{ for all } 1 \leq i\ne j \leq n \}.
\end{array}
\nonumber
\end{align}

The above relations imply that if for a given $\mathbf{w}>\mathbf{0}$, there exists an index pair $(k,\ell), k\ne\ell$, such that  $a_{k\ell}= \frac{w_k}{w_\ell}$, then $\mathbf{w}\in W\!E$,  $\mathbf{w}\in W\!E_L$ and  $\mathbf{w}\in W\!E_I$.

It is evident that $E\subseteqq W\!E$,
$E_L\subseteqq W\!E_L$ and $E_I\subseteqq W\!E_I$. We show below that the relations $E=E_L=E_I$ and $W\!E=W\!E_L=W\!E_I$ also hold. This means that the   three def{\kern0pt}initions given, regarding both the stronger and the weaker cases of ef{\kern0pt}f{\kern0pt}iciency, are equivalent.
Example \ref{example:0a} demonstrates that $E\subsetneqq W\!E$.

For a given $\mathbf{w}>\mathbf{0}$, let $D(\mathbf{w})$ denote the set of the points \textit{dominating} the point $\mathbf{w}$, i.e.
\begin{align}
D(\mathbf{w}) = \{\mathbf{x}>\mathbf{0}\mid &
f_{ij}(\mathbf{x})\le f_{ij}(\mathbf{w})\text{ for all } i\ne j \text{ and }\nonumber\\
&f_{k\ell}(\mathbf{x})< f_{k\ell}(\mathbf{w})\text{ for some } k\ne \ell \}.
\nonumber
\end{align}
Similarly, let $SD(\mathbf{w})$ denote the set of the points \textit{strongly dominating} the point $\mathbf{w}$, i.e.
\begin{align}
SD(\mathbf{w}) = \{\mathbf{x}>\mathbf{0}\mid
f_{ij}(\mathbf{x}) < f_{ij}(\mathbf{w})\text{ for all } i\ne j  \}.
\nonumber
\end{align}
It is easy to see that if $SD(\mathbf{w})\ne\emptyset$, then $SD(\mathbf{w}) =\text{int}(D(\mathbf{w}))$ and $\text{cl}(SD(\mathbf{w})) =\text{cl}(D(\mathbf{w}))$, where int and cl denote, the interior and closure, respectively, of the relating set.

\begin{proposition} \label{proposition:DandSDareConvex}
$D(\mathbf{w})$ and $SD(\mathbf{w})$ are convex sets, and if any of them is nonempty, then $\mathbf{w}$ lies in its boundary.
\end{proposition}
\begin{proof}
We start with the proof of the simpler case of  $SD(\mathbf{w})$. Clearly,
\begin{align}
\mathbf{x}\in SD(\mathbf{w})
 \Longleftrightarrow
 &\left| \frac{x_i}{x_j} - a_{ij}\right| < f_{ij}(\mathbf{w})\ \text{ for all } i\ne j
 \Longleftrightarrow  \nonumber  \\
 &\frac{x_i}{x_j} - a_{ij} < f_{ij}(\mathbf{w}),\ -\frac{x_i}{x_j} + a_{ij} < f_{ij}(\mathbf{w})\ \text{ for all } i\ne j
 \Longleftrightarrow  \nonumber \\
 &x_i + (- a_{ij}-f_{ij}(\mathbf{w}))x_j< 0,\
 -x_i + (a_{ij}-f_{ij}(\mathbf{w}))x_j< 0,
 \ \text{ for all } i\ne j.
 \nonumber
\end{align}
The set of points fulf{\kern0pt}illing the last system of strict inequalities is an intersection of f{\kern0pt}initely many open halfspaces, it is thus an open convex set. At the same time, with $\mathbf{x} = \mathbf{w}$, the linear inequalities above hold as equalities, consequently, $\mathbf{w}$ lies in the boundary of $SD(\mathbf{w})$, of course, if it is nonempty.

By applying similar rearranging steps, we also get that
\begin{align}
\mathbf{x}\in D(\mathbf{w})
 \Longleftrightarrow
 &x_i + (- a_{ij}-f_{ij}(\mathbf{w}))x_j\le 0,\
 -x_i + (a_{ij}-f_{ij}(\mathbf{w}))x_j\le 0,
 \ \text{ for all } i\ne j, \text{ and}
 \label{eq:Dw1} 
  \\
  &x_k + (- a_{k\ell}-f_{k\ell}(\mathbf{w}))x_\ell< 0,\
 -x_k + (a_{k\ell}-f_{k\ell}(\mathbf{w}))x_\ell< 0,
 \text{ for some } k\ne \ell.
 \label{eq:Dw2} 
\end{align}

We show that $D(\mathbf{w})$ is a convex set.
Let $\mathbf{y}\ne\mathbf{z}\in D(\mathbf{w})$, $0<\lambda<1$ and
${\mathbf{\hat x}}=\lambda\mathbf{y} +
(1-\lambda )\mathbf{z}$.
The linear inequalities of (\ref{eq:Dw1}) hold at the points ${\mathbf{\hat x}}, \mathbf{y}$ and $\mathbf{z}$.

Let $(\hat{k},\hat{\ell)}, \hat{k}\ne\hat{\ell}$ denote the index pair for which (\ref{eq:Dw2}) also holds at the point $\mathbf{x}=\mathbf{y}$.
Then, with $\mathbf{x}=\mathbf{\hat x}$, (\ref{eq:Dw2}) also holds for the index pair $(\hat{k},\hat{\ell)}$.
This implies $\mathbf{\hat x}\in D(\mathbf{w})$ and the convexity of $D(\mathbf{w})$.

The point $\mathbf{x}=\mathbf{w}$ fulf{\kern0pt}ils (\ref{eq:Dw1}) as equalities. Thus, $\mathbf{w}\not\in D(\mathbf{w})$ but $\mathbf{w}$ is boundary point of $D(\mathbf{w})$ if it is nonempty.
\end{proof}

\medskip
\begin{proposition} \label{proposition:EandWEandEquivalentForms}
$E=E_L=E_I$ and $W\!E=W\!E_L=W\!E_I$.
\end{proposition}
\begin{proof}
Obviously,
$$
\begin{array}{lllll}
E\ \ =\{\mathbf{w}>\mathbf{0}\mid D(\mathbf{w})=\emptyset\},\\
E_L=\{\mathbf{w}>\mathbf{0}\mid D(\mathbf{w})\cap U(\mathbf{w})=\emptyset\},
\text{ where } U(\mathbf{w})
\text{ is a suitably small neighborhood}\\ ~~~~~~~~\text{ around } \mathbf{w},\text{ and }\\
E_I=\{\mathbf{w}>\mathbf{0}\mid D(\mathbf{w})\cap V_I(\mathbf{w})=\emptyset\},
\text{ where } \\
~~~~~~~~V_I(\mathbf{w})=\{\mathbf{x}>\mathbf{0}\mid a_{ij}\le \frac{x_i}{x_j}\le  \frac{w_i}{w_j}\text{ for }  a_{ij}\le   \frac{w_i}{w_j}, \forall i\ne j,
 a_{ij}\ge \frac{x_i}{x_j}\ge  \frac{w_i}{w_j}\text{ for }  a_{ij}\ge   \frac{w_i}{w_j}, \forall i\ne j \}\\
~~~~~~~~\text{ is a convex set containing } \mathbf{w}.
\end{array}
$$

If $\mathbf{w}\in E$, then $D(\mathbf{w})=\emptyset$, thus $\mathbf{w}\in E_L$ and $\mathbf{w}\in E_I$, therefore, $E\subseteq E_L$ and $E\subseteq E_I$.

We show that $E_L\subseteq E$ holds, as well. Let $\mathbf{w}\in E_L$ and assume that $\mathbf{w}\not\in E$, i.e.\ $D(\mathbf{w})\ne\emptyset$. Let $\mathbf{\hat{x}}\in D(\mathbf{w})$.  Since
$\mathbf{w}$ is a boundary point of the convex set $D(\mathbf{w})$, every point of the half-open line segment $[\mathbf{\hat{x}}, \mathbf{w})$ is in $D(\mathbf{w})$. However, the points of $[\mathbf{\hat{x}}, \mathbf{w})$ that are close enough to $\mathbf{w}$ are also in $U(\mathbf{w})$. This contradicts
$\mathbf{w}\in E_L$ since $D(\mathbf{w})\cap U(\mathbf{w})\ne\emptyset$. Consequently, $\mathbf{w}\in E$, thus,
$E_L\subseteq E$, and then $E_L = E$ .

The proof of $E_I\subseteq E$ is similar.
Let $\mathbf{w}\in E_I$ and assume that $\mathbf{w}\not\in E$. Let $\mathbf{\hat{x}}\in D(\mathbf{w})$. Now, if $a_{ij}=\frac{w_i}{w_j}$, then also
$a_{ij}=\frac{x_i}{x_j}$ for every $\mathbf{x}\in [\mathbf{\hat{x}},\mathbf{w}]$. If  $a_{ij}<\frac{w_i}{w_j}$, then
$a_{ij}<\frac{x_i}{x_j}<\frac{w_i}{w_j}$ for the points $\mathbf{x}\in [\mathbf{\hat{x}},\mathbf{w}]$ being close enough  to $\mathbf{w}$. The same holds in case of  $a_{ij}>\frac{w_i}{w_j}$ with opposite sign.  These imply that $[\mathbf{\hat{x}},\mathbf{w})\cap D(\mathbf{w})\cap V_I(\mathbf{w})\ne\emptyset$, leading again to a contradiction. Consequently,
$E_I\subseteq E$, so $E_I = E$ .

The proof of the relations $W\!E=W\!E_L=W\!E_I$ can be carried out in the same way, we have simply use the set $SD(\mathbf{w})$ instead of $D(\mathbf{w})$. The remainder part of the proof is left to the reader.
\end{proof}

\section{The principal right eigenvector is weakly ef{\kern0pt}f{\kern0pt}icient}
\label{section:3} 

Blanquero, Carrizosa and Conde \cite[p.~279]{BlanqueroCarrizosaConde2006} stated (without proof) that
weak ef{\kern0pt}f{\kern0pt}iciency is equivalent to that the directed graph, according to Def{\kern0pt}inition \ref{def:DirectedGraphEfficient}, includes at least one cycle. Here we rephrase
the proposition and give a proof.

\begin{lemma} \label{lemma:acyclic}   
Let $\mathbf{A}$ be an arbitrary pairwise comparison matrix of size $n \times n$
and $\mathbf{w}$ be an arbitrary positive weight vector.
Weight vector $\mathbf{w}$ is strongly inef{\kern0pt}f{\kern0pt}icient
 for (\ref{eq:xMultiObjectiveOptimizationProblem})
if and only if
its digraph is isomorphic to the acyclic tournament on $n$ vertices (including
$arc(i,j)$ if and only if $i < j$).
\end{lemma}
\begin{proof}
Suf{\kern0pt}f{\kern0pt}iciency. Assume without loss of generality that the rows and columns of pairwise comparison matrix $\mathbf{A}$
are permuted such that digraph $G$ includes $arc(i,j)$ if and only if $i < j.$
Then
\begin{align}
\frac{w_i}{w_j} > a_{ij} \qquad \text{ for all } 1 \leq i < j \leq n,
\label{def:AcyclicInequality}    
\end{align}
and, equivalently, $\frac{w_i}{w_j} < a_{ij} $ for all
$ 1 \leq j < i \leq n.$ We shall f{\kern0pt}ind a weight vector
$\mathbf{w}^{\prime}$ such that
(\ref{def:DefinitionEfficientProperty1}), moreover,
$\frac{w_i}{w_j} > \frac{w^{\prime}_i}{w^{\prime}_j} \geq a_{ij} $
hold for all $ 1 \leq i < j \leq n.$

Let
\begin{equation}
p_j := \underset{i=1,2,\ldots,j-1}{\max} \left\{  \frac{a_{ij}}{\frac{w_i}{w_j}} \right\},
\, j=2,3,\ldots,n. \label{multipliers}    
\end{equation}
It follows from (\ref{def:AcyclicInequality}) that
$ p_j < 1$  for all $ 2 \leq j \leq n.$
Let ${w}^{\prime}_k := w_k \cdot \prod\limits_{j=k+1}^{n} p_j$
for all $ 1 \leq k \leq n-1,$ and ${w}^{\prime}_n := w_n.$
It follows from the construction that
\[
\frac{{w}^{\prime}_k}{{w}^{\prime}_{\ell}} =
\frac{w_k}{w_{\ell}} \frac{\prod\limits_{j=k+1}^{n} p_j}{\prod\limits_{j=\ell+1}^{n} p_j} =
\frac{w_k}{w_{\ell}} \prod\limits_{j=k+1}^{\ell} p_j < \frac{w_k}{w_{\ell}}.
\]
On the other hand, (\ref{multipliers}) ensures that $\frac{{w}^{\prime}_k}{{w}^{\prime}_{\ell}} \geq a_{k\ell}.$
Furthermore, for every $ 1 \leq k \leq n-1$  there exists a (not necessarily unique) $\ell > k$ such that
$\frac{{w}^{\prime}_k}{{w}^{\prime}_{\ell}} = a_{k\ell}.$
Especially $\frac{{w}^{\prime}_1}{{w}^{\prime}_{2}} = a_{12}.$
\\

For necessity let us suppose that digraph $G$ includes a directed 3-cycle $(i,j,k)$:
$
\frac{w_i}{w_j} > a_{ij},
\frac{w_j}{w_k} > a_{jk},
\frac{w_k}{w_i} > a_{ki}.
$
Assume for contradiction that weight vector $\mathbf{w}$ is strongly inef{\kern0pt}f{\kern0pt}icient,
that is, there exists another weight vector $\mathbf{w}^{\prime}$ such that
(\ref{def:DefinitionEfficientProperty1}) holds.
Then
\newpage
\begin{align}
\frac{w_i}{w_j}  &>  \frac{w_i^{\prime}}{w_j^{\prime}},  \label{wiwjaij}  \\  
\frac{w_j}{w_k}  &>  \frac{w_j^{\prime}}{w_k^{\prime}},  \label{wjwkajk}  \\  
\frac{w_k}{w_i}  &>  \frac{w_k^{\prime}}{w_i^{\prime}},  \label{wkwiaki}    
\end{align}
otherwise none of
\[
\left| \frac{w_i^{\prime}}{w_j^{\prime}} - a_{ij}  \right|  <  \left| \frac{w_i}{w_j} - a_{ij}  \right|,  \quad
\left| \frac{w_j^{\prime}}{w_k^{\prime}} - a_{jk}  \right|  <  \left| \frac{w_j}{w_k} - a_{jk}  \right|,  \quad
\left| \frac{w_k^{\prime}}{w_i^{\prime}} - a_{ki}  \right|  <  \left| \frac{w_k}{w_i} - a_{ki}  \right|
\]
could hold. Multiply inequalities (\ref{wiwjaij})-(\ref{wkwiaki})
to get the contradiction $1>1.$
\end{proof}

\begin{corollary} \label{corollary:outdegrees}
Weight vector is strongly inef{\kern0pt}f{\kern0pt}icient
 for (\ref{eq:xMultiObjectiveOptimizationProblem})
if and only if the set of outdegrees
in the associated directed graph is $\{0,1,2,\ldots,n-1\}.$
\end{corollary}

\begin{theorem}
The principal eigenvector of a pairwise comparison matrix is weakly ef{\kern0pt}f{\kern0pt}icient
 for (\ref{eq:xMultiObjectiveOptimizationProblem}).
\end{theorem}
\begin{proof}
The principal right eigenvector $\mathbf{w}$ satisf{\kern0pt}ies the equation
\begin{equation}
\mathbf{A} \mathbf{w} = \lambda_{\max} \mathbf{w}. \label{Avlambdav}    
\end{equation}
Assume for contradiction that weight vector $\mathbf{w}$ is strongly inef{\kern0pt}f{\kern0pt}icient.
Apply Lemma \ref{lemma:acyclic}
and consider the acyclic tournament associated to $\mathbf{A}$ and $\mathbf{v}$.
We can assume without loss of generality that the Hamiltonian path is already
$1 \rightarrow 2 \rightarrow \ldots \rightarrow n.$
Then
\begin{equation}
 \frac{w_i}{w_j} > a_{ij} \text{ for all } 1 \leq i < j \leq n.  \label{vivjaij}    
\end{equation}
The $i$-th equation of (\ref{Avlambdav}) is
\begin{equation}
\sum\limits_{j=1}^{n} a_{ij} w_j = \lambda_{\max} w_i,  \label{Avlambdav-i}    
\end{equation}
the left hand side of which is bounded above due to (\ref{vivjaij}):
\[
\sum\limits_{j=1}^{n} a_{ij} w_j <  \sum\limits_{j=1}^{n} \frac{w_i}{w_j} w_j = n w_i
\]
which contradicts $\lambda_{\max} \geq n.$
\end{proof}

\section{Ef{\kern0pt}f{\kern0pt}iciency test and search for an ef{\kern0pt}f{\kern0pt}icient dominating weight vector by linear programming}
\label{section:4} 

Let a pairwise comparison matrix
$\mathbf{A} =
\left[
a_{ij}
\right]_{i,j=1,\ldots,n}$ and a positive weight vector
$\mathbf{w} = (w_1, w_2, \ldots, w_n)^{\top}$ be given as before.
First we shall verify whether $\mathbf{w}$ is ef{\kern0pt}f{\kern0pt}icient
for (\ref{eq:xMultiObjectiveOptimizationProblem}) by solving an appropriate linear program.
Furthermore, if $\mathbf{w}$ is inef{\kern0pt}f{\kern0pt}icient, the optimal solution
of the linear program provides an ef{\kern0pt}f{\kern0pt}icient weight vector that dominates $\mathbf{w}$ internally. \\

Recall the double inequality (\ref{def:DefinitionInternallyEfficientProperty1}) in Def{\kern0pt}inition
(\ref{def:DefinitionInternallyEfficient}). For every positive weight vector
$\mathbf{x}=(x_1,x_2,\ldots,x_n)^{\top}$
\begin{equation}
\label{aijxixjeq1} 
\begin{aligned}
a_{ij} \leq \frac{x_i}{x_j} \underset{(<)}{\leq} \frac{w_i}{w_j} \, &\Longleftrightarrow \,
\left( \frac{a_{ij}x_j}{x_i}  \leq 1, \, \frac{x_i}{x_j} \frac{w_j}{w_i} \underset{(<)}{\leq} 1 \right) \, \Longleftrightarrow \\
&\Longleftrightarrow \, \left( \frac{a_{ij}x_j}{x_i}  \leq 1, \, \frac{x_i}{x_j} \frac{w_j}{w_i} \frac{1}{t_{ij}} \leq 1 \text{ for some }
0 < t_{ij} \underset{(<)}{\leq} 1 \right),
\end{aligned}
\end{equation}
and
\begin{equation}
\label{aijxixjeq2}  
\begin{aligned}
a_{ij} \geq \frac{x_i}{x_j} \underset{(>)}{\geq} \frac{w_i}{w_j} \, &\Longleftrightarrow \,
\left( \frac{x_i}{a_{ij}x_j}  \leq 1, \, \frac{x_j}{x_i} \frac{w_i}{w_j} \underset{(<)}{\leq} 1 \right) \, \Longleftrightarrow \\
&\Longleftrightarrow \, \left( \frac{x_i}{a_{ij}x_j}  \leq 1, \, \frac{x_j}{x_i} \frac{w_i}{w_j} \frac{1}{t_{ij}} \leq 1 \text{ for some }
0 < t_{ij} \underset{(<)}{\leq} 1 \right),
\end{aligned}
\end{equation}
and
\begin{equation}
\label{aijxixjeq3}  
\begin{aligned}
a_{ij}  = \frac{x_i}{x_j}  \, \Longleftrightarrow \, \frac{x_i}{a_{ij}x_j}  = 1.
\end{aligned}
\end{equation}
This leads us to develop the following optimization problem.\\

Def{\kern0pt}ine index sets
\begin{align}
 I &= \left\{ (i,j) \left|  a_{ij} < \frac{w_i}{w_j}   \right. \right\} \nonumber  \\
 J &= \left\{ (i,j) \left|  a_{ij} = \frac{w_i}{w_j}, i < j   \right. \right\} \nonumber
\end{align}
The index set $I$ is empty if and only if pairwise comparison matrix $\mathbf{A}$ is consistent. In this case
weight vector $\mathbf{w}$ is ef{\kern0pt}f{\kern0pt}icient and $|J| = n(n-1)/2$.
It is assumed in the sequel that $I$ is nonempty. No assumptions are needed for the (non)emptiness of $J$.

\begin{align}
\min  \prod\limits_{(i,j)\in I } t_{ij}
     &  &  \nonumber  \\
\frac{x_j}{x_i} a_{ij} & \leq 1    &\text{ for all $(i,j) \in I$,} \nonumber  \\
\frac{x_i}{x_j} \frac{w_j}{w_i} \frac{1}{t_{ij}}  & \leq 1  &\text{ for all $(i,j) \in I$,} \label{P}  \\
0 < t_{ij}  & \leq 1 &\text{ for all $(i,j) \in I$,} \nonumber  \\
a_{ji} \frac{x_i}{x_j} & = 1    &\text{ for all $(i,j) \in J$,} \nonumber \\
                   x_1 & = 1    & \nonumber
\end{align}
Variables are $x_i > 0, \,  1 \leq i \leq n$ and $t_{ij}, \, (i,j) \in I.$ \\

\begin{proposition} \label{prop:P}
The optimum value of the optimization problem (\ref{P}) is at most 1 and it is equal to 1 if and only if
weight vector $\mathbf{w}$ is ef{\kern0pt}f{\kern0pt}icient
 for (\ref{eq:xMultiObjectiveOptimizationProblem}).
Denote the optimal solution to (\ref{P}) by $(\mathbf{x}^{\ast},\mathbf{t}^{\ast}) \in \mathbb{R}^{n+|I|}_{+}.$
If weight vector $\mathbf{w}$ is inef{\kern0pt}f{\kern0pt}icient, then
weight vector $\mathbf{x}^{\ast}$ is ef{\kern0pt}f{\kern0pt}icient and dominates $\mathbf{w}$ internally.
\end{proposition}

\begin{proof}
The constraints in (\ref{aijxixjeq1})-(\ref{aijxixjeq3}) are obtained by simple rearrangements. It is obvious that in (\ref{aijxixjeq1})
 $\frac{x_i}{x_j} \frac{w_j}{w_i} \leq 1$
if and only there exists a scalar
$0 < t_{ij} {\leq} 1$ such that
$\frac{x_i}{x_j} \frac{w_j}{w_i} \frac{1}{t_{ij}}\le 1$. In addition, the inequalities hold as strict inequalities simultaneously on both sides. The reasoning is similar for (\ref{aijxixjeq2}), and (\ref{aijxixjeq3}) is evident.

In (\ref{P}), only the constraints belonging to the index pairs from $I$ and $J$ appear. Due to the reciprocity property, the remainder constraints are now redundant. First, we show that the feasible set of problem (\ref{P}) is a nonempty compact set. Therefore, since the objective function is continuous, (\ref{P}) has a f{\kern0pt}inite optimum value and an optimal solution.

Problem (\ref{P}) has a feasible solution, e.g. $\mathbf{x}=\frac{1}{w_1}\mathbf{w}$ and $t_{ij}=1,\,\text{ for all } (i,j)\in I$ fulf{\kern0pt}ill the constraints. Due to the normalization constraint $x_1=1$, the other variables $x_i, i\ne 1$, have  f{\kern0pt}inite positive lower and upper bounds over the feasible set. This comes from the property that for all $i\ne 1$, either $(i,1)$ or $(1,i)$ is in $I\cup J$. The fourth constraint gives a f{\kern0pt}ixed value for $x_i$, and  positive lower and upper bounds can be computed from the f{\kern0pt}irst and second constraints. Since the components of $\mathbf{x}$ have positive upper and lower bound, from the second constraint, positive lower bounds can be computed for the variables $t_{ij}, (i,j)\in I$, too.

The objective function serves for testing the internal ef{\kern0pt}f{\kern0pt}iciency  of  $\mathbf{w}$. Its value cannot exceed 1.
If its value is less than 1, then there exists an index pair
$(i_0, j_0)$ for which 
$\frac{x_{i_0}}{x_{j_0}} \frac{w_{j_0}}{w_{i_0}}\le t_{{i_0}{j_0}}<1$,
hence $\frac{x_{i_0}}{x_{j_0}}<\frac{w_{i_0}}{w_{j_0}}$. From this and the equivalent forms in (\ref{aijxixjeq1}) and (\ref{aijxixjeq3}), we get that $\mathbf{x}$ internally dominates $\mathbf{w}$. Conversely, assume that $\mathbf{x}$ internally dominates $\mathbf{w}$. It is easy to see that the normalized vector $\mathbf{x}$ with
$t_{{i}{j}} = \frac{x_{i}}{x_{j}} \frac{w_{j}}{w_{i}},\,(i,j)\in I$, is feasible to (\ref{P}). In addition, for every index pair $(i_0, j_0)$ for which, due to the internal dominance,  $\frac{x_{i_0}}{x_{j_0}}<\frac{w_{i_0}}{w_{j_0}}$ holds, we have $t_{{i_0}{j_0}}<1$, thus, the considered feasible solution has an objective function value less than 1, implying that the optimal value is also less than 1.

It remains to deal with the case when $\mathbf{w}$ turns out to be inef{\kern0pt}f{\kern0pt}icient. It is obvious that the $\mathbf{x}$-part of the optimal solution $(\mathbf{x}^{\ast},\mathbf{t}^{\ast})$ of (\ref{P}) internally dominates $\mathbf{w}$, and $t_{{i}{j}}^{\ast}= \frac{x_{i}^{\ast}}{x_{j}^{\ast}} \frac{w_{j}}{w_{i}},\,\text{ for all } (i,j)\in I$. Assume that $\mathbf{x}^{\ast}$ is inef{\kern0pt}f{\kern0pt}icient. Then it is internally dominated by a vector $\mathbf{\bar x}$, For $\mathbf{\bar x}$, we have $a_{ij}=\frac{\bar x_i}{\bar x_j},\,\text{ for all } (i,j)\in J$. Also,
$a_{ij}\le\frac{\bar x_i}{\bar x_j}\le\frac{x_i^{\ast}}{{x_j}^{\ast}}\le\frac{ w_i}{w_j},\,\text{ for all } (i,j)\in I$ and there exists at least one index pair
$(i_0, j_0)\in I$
for which the second inequality is strict.   Let
$\bar t_{{i}{j}}= \frac{\bar x_{i}}{\bar x_{j}} \frac{w_{j}}{w_{i}},\,\text{ for all } (i,j)\in I$.
It is easy to see that after a normalization, $(\mathbf{\bar x},\mathbf{\bar t})$ is feasible to (\ref{P}). However, we also have $\bar t_{{i}{j}}\le t_{{i}{j}}^{\ast}, \,\text{ for all } (i,j)\in I$ and
$\bar t_{{i_0}{j_0}} < t_{{i_0}{j_0}}^{\ast}$. This implies that the objective function value at $(\mathbf{\bar x},\mathbf{\bar t})$ is less than that at $(\mathbf{x}^{\ast},\mathbf{t}^{\ast})$. This contradicts the fact that
$(\mathbf{x}^{\ast},\mathbf{t}^{\ast})$ is an optimal solution to (\ref{P}). Consequently, $\mathbf{x}^{\ast}$ is an ef{\kern0pt}f{\kern0pt}icient solution.

\end{proof}

Optimization problem (\ref{P}) is nonlinear but it can be transformed to a linear program.
Let denote
$y_i = \log x_i,  \, v_i = \log w_i, \, 1 \leq i \leq n;$
$s_{ij} = - \log t_{ij}, \, (i,j) \in I;$ and
$b_{ij} = \log a_{ij}, \, 1 \leq i,j \leq n.$
Taking the logarithm of the objective function and the constraints in (\ref{P}), we arrive at an equivalent linear program

\begin{align}
\min  \sum\limits_{(i,j)\in I } - s_{ij}
     &  &  \label{LP1}  \\
y_j - y_i  & \leq - b_{ij}    &\text{ for all $(i,j) \in I$,} \label{LP2}  \\
y_i - y_j + s_{ij} & \leq  v_i - v_j  &\text{ for all $(i,j) \in I$,} \label{LP3} \\
y_i - y_j  & = b_{ij}    &\text{ for all $(i,j) \in J$,} \label{LP4} \\
s_{ij}  & \geq 0  &\text{ for all $(i,j) \in I$,} \label{LP5} \\
                   y_1 & = 0    & \label{LP6}
\end{align}
Variables are $y_i, \,  1 \leq i \leq n$ and $s_{ij} \geq 0, \, (i,j) \in I.$ \\

\begin{theorem} \label{thm:LP}   
The optimum value of the linear program (\ref{LP1})-(\ref{LP6}) is at most 0
and it is equal to 0 if and only if weight vector $\mathbf{w}$ is ef{\kern0pt}f{\kern0pt}icient
 for (\ref{eq:xMultiObjectiveOptimizationProblem}).
Denote the optimal solution to (\ref{LP1})-(\ref{LP6}) by $(\mathbf{y}^{\ast},\mathbf{s}^{\ast}) \in \mathbb{R}^{n+|I|}.$
If weight vector $\mathbf{w}$ is inef{\kern0pt}f{\kern0pt}icient, then
weight vector $\exp({\mathbf{y}}^{\ast})$ is ef{\kern0pt}f{\kern0pt}icient and dominates $\mathbf{w}$ internally.
\end{theorem}

An example is given in the Appendix.

\section{Test of weak ef{\kern0pt}f{\kern0pt}iciency and search for an ef{\kern0pt}f{\kern0pt}icient dominating weight vector by linear programming}
\label{section:5} 

The test of weak ef{\kern0pt}f{\kern0pt}iciency and searching for a dominating weakly ef{\kern0pt}f{\kern0pt}icient point can be carried out similarly to the case of ef{\kern0pt}f{\kern0pt}iciency. Consider a vector $\mathbf{w}>\mathbf{0}$. Obviously, if $J\ne\emptyset$, i.e.\ $f_{ij}(\mathbf{w})=0$ for an index pair $i\ne j$, then $\mathbf{w}\in W\!E$, so we are ready with the test of weak ef{\kern0pt}f{\kern0pt}iciency.

Now, examine the case $J=\emptyset$. Then $|I| = n(n-1)/2$.
Note that if the rows and columns of pairwise comparison matrix $\mathbf{A}$
are permuted according to Lemma \ref{lemma:acyclic}, then $I = \{(i,j) | 1 \leq i < j \leq n  \}.$
Here are some equivalent forms for strong inef{\kern0pt}f{\kern0pt}iciency. For all $(i,j) \in I$

\begin{align}
a_{ij} \leq \frac{x_i}{x_j} < \frac{w_i}{w_j}   \Longleftrightarrow &  \,
\left( \frac{a_{ij}x_j}{x_i}  \leq 1, \,
\frac{x_i}{x_j} \frac{w_j}{w_i} < 1 \right)\ \nonumber
\\
\Longleftrightarrow & \, \left( \frac{a_{ij}x_j}{x_i}  \leq 1, \, \frac{x_i}{x_j} \frac{w_j}{w_i} \frac{1}{t} \leq 1,
0 < t < 1 \right) . 
\label{eq:WEforms} 
\end{align}

Based on the last form of (\ref{eq:WEforms}), we can establish a modif{\kern0pt}ication of problem (\ref{P}), adapting it to the case of weak ef{\kern0pt}f{\kern0pt}iciency.
\begin{align}
\min t
     &    \nonumber  \\
\frac{x_j}{x_i} a_{ij} & \leq 1 \ \text{ for all $(i,j) \in I$,} \nonumber  \\
\frac{x_i}{x_j} \frac{w_j}{w_i} \frac{1}{t}  & \leq 1 \ \text{ for all $(i,j) \in I$,} \label{eq:WP}  \\
0 < t & \leq 1  \nonumber  \\                   x_1 & = 1.    \nonumber
\end{align}
Variables are $x_i > 0, \,  1 \leq i \leq n$ and $t$. \\

\begin{proposition} \label{prop:WP}
The optimum value of the optimization problem (\ref{eq:WP}) is at most 1 and it is equal to 1 if and only if
weight vector $\mathbf{w}$ is weakly ef{\kern0pt}f{\kern0pt}icient
 for (\ref{eq:xMultiObjectiveOptimizationProblem}).
Denote the optimal solution to (\ref{eq:WP}) by $(\mathbf{x}^{\ast},t^{\ast}) \in \mathbb{R}^{n+1}_{+}.$
If weight vector $\mathbf{w}$ is strongly inef{\kern0pt}f{\kern0pt}icient, then
weight vector $\mathbf{x}^{\ast}$ is weakly ef{\kern0pt}f{\kern0pt}icient and dominates $\mathbf{w}$ internally and strictly.
\end{proposition}

\begin{proof}
The statements can be proved by analogy with the proof of Proposition \ref{prop:P}. By using the same reasoning as there, one can easily show that the feasible set of (\ref{eq:WP}) is not empty, and positive upper and lower bounds can be determined for each variable. Thus (\ref{eq:WP}) has an optimal solution and a positive optimal value $t^{\ast}\le 1$.

If $t^{\ast} < 1$, then
$\frac{x_{i}}{x_{j}} \frac{w_{j}}{w_{i}}\le t^{\ast} < 1$ for all $i \neq j$, implying that $\mathbf{x}$ internally strongly dominates $\mathbf{w}$. Conversely, assume that $\mathbf{x}$ internally strongly dominates $\mathbf{w}$. It is easy to see that the normalized vector $\mathbf{x}$ with
$t = \max_{i \neq j}\frac{x_{i}}{x_{j}} \frac{w_{j}}{w_{i}}$ is feasible to (\ref{eq:WP}). It is obvious that $0<t<1$ at this feasible solution, implying that $t^{\ast} < 1$ at the optimal solution.

Consider the case when $\mathbf{w}$ has turned out to be weakly inef{\kern0pt}f{\kern0pt}icient, i.e. it is strongly dominated.  It is obvious that the $\mathbf{x}$-part of the optimal solution $(\mathbf{x}^{\ast},\mathbf{t}^{\ast})$ of (\ref{eq:WP}) internally dominates $\mathbf{w}$, and $t^{\ast}=\max_{i \neq j} \frac{x_{i}^{\ast}}{x_{j}^{\ast}} \frac{w_{j}}{w_{i}}$. Assume that $\mathbf{x}^{\ast}$ is inef{\kern0pt}f{\kern0pt}icient. Then it is internally strongly dominated by a vector $\mathbf{\bar x}$. For $\mathbf{\bar x}$,
$a_{ij}\le\frac{\bar x_i}{\bar x_j} < \frac{x_i^{\ast}}{{x_j}^{\ast}}\le\frac{ w_i}{w_j},\, \text{ for all } i \neq j$.   Let
$\bar t=\max_{i \neq j} \frac{\bar x_{i}}{\bar x_{j}} \frac{w_{j}}{w_{i}}$.
It is easy to see that after a normalization, $(\mathbf{\bar x},\mathbf{\bar t})$ is feasible to (\ref{eq:WP}). It is however obvious that  $\bar t < t^{\ast}$, implying that the objective function value at $(\mathbf{\bar x},{\bar t})$ is less than that at $(\mathbf{x}^{\ast},{t}^{\ast})$ contradicting the optimality of the latter one. Consequently, $\mathbf{x}^{\ast}$ is a weakly  ef{\kern0pt}f{\kern0pt}icient solution.
\end{proof}

By using the same idea that was applied to get problem  (\ref{LP1})-(\ref{LP6}) from (\ref{P}), problem (\ref{eq:WP}) can also be transformed to a linear program.
By using the same notations as there, and introducing the variable
$s = - \log t$, we arrive at an equivalent linear program

\begin{align}
\min - s
     &   \nonumber  \\
y_j - y_i  & \leq - b_{ij} \ \ \ \ \text{ for all $(i,j) \in I$} \nonumber  \\
y_i - y_j + s & \leq  v_i - v_j  \ \text{ for all $(i,j) \in I$} \label{eq:WLP} \\ 
s  & \geq 0  \nonumber \\
                   y_1 & = 0     \nonumber
\end{align}
Variables are $y_i, \,  1 \leq i \leq n$ and $s$. \\

\begin{theorem} \label{thm:WLP} 
The optimum value of the linear program (\ref{eq:WLP}) is at most 0
and it is equal to 0 if and only if weight vector $\mathbf{w}$ is weakly ef{\kern0pt}f{\kern0pt}icient
 for (\ref{eq:xMultiObjectiveOptimizationProblem}).
Denote the optimal solution to (\ref{eq:WLP}) by $(\mathbf{y}^{\ast},s) \in \mathbb{R}^{n+1}.$
If weight vector $\mathbf{w}$ is strongly inef{\kern0pt}f{\kern0pt}icient, then
weight vector $\exp({\mathbf{y}}^{\ast})$ is weakly ef{\kern0pt}f{\kern0pt}icient and dominates $\mathbf{w}$ internally and strictly.
\end{theorem}

\begin{remark}
If weight vector $\mathbf{w}$ is strongly inef{\kern0pt}f{\kern0pt}icient
 for (\ref{eq:xMultiObjectiveOptimizationProblem}),
then
weight vector $\exp({\mathbf{y}}^{\ast})$ in Theorem \ref{thm:WLP} is weakly ef{\kern0pt}f{\kern0pt}icient, but not necessarily
ef{\kern0pt}f{\kern0pt}icient. However, linear program (\ref{LP1})-(\ref{LP6}) in Section \ref{section:4} can test its ef{\kern0pt}f{\kern0pt}iciency, and it if is inef{\kern0pt}f{\kern0pt}icient,
(\ref{LP1})-(\ref{LP6}) f{\kern0pt}ind a dominating ef{\kern0pt}f{\kern0pt}icient weight vector, that obviously dominates (internally and strictly)
the strongly inef{\kern0pt}f{\kern0pt}icient weight vector $\mathbf{w}$, too.
\end{remark}

\newpage
\section{Conclusions and open questions}
\label{section:6} 

\subsection{Conclusions}
The key problem of weighting is to approximate the elements of a pairwise comparison matrix, f{\kern0pt}illed in
by the decision maker.
The multi-objective optimization problem (\ref{eq:xMultiObjectiveOptimizationProblem}) has a unique
solution only for consistent pairwise comparison matrices.
Numerical examples show that certain weighting methods, such as the eigenvector, result in inef{\kern0pt}f{\kern0pt}icient
(for (\ref{eq:xMultiObjectiveOptimizationProblem})) solutions. Less formally, the pairwise ratios
do not approximate the matrix elements
in the \emph{best} possible way,
since some of the estimations can be strictly improved without impairing any other one.
Nevertheless, the weak ef{\kern0pt}f{\kern0pt}iciency of the principal eigenvector has been proved in Section 3.

Linear programs have been developed in Sections \ref{section:4} and \ref{section:5} in order to
test ef{\kern0pt}f{\kern0pt}iciency, and to f{\kern0pt}ind an ef{\kern0pt}f{\kern0pt}icient dominating weight vector.

\subsection{Open questions}
Ef{\kern0pt}f{\kern0pt}iciency for (\ref{eq:xMultiObjectiveOptimizationProblem}) is a potential
criterion in future comparative studies of the weighting methods.
Our opinion is that an inef{\kern0pt}f{\kern0pt}icient weight vector is less preferred to any of its
dominating weight vectors, and, especially to the ef{\kern0pt}f{\kern0pt}icient dominating weight vector(s).

The use of models developed in Sections \ref{section:4} and
\ref{section:5} enables the decision maker to improve a possibly inef{\kern0pt}f{\kern0pt}icient weight vector,
however, the problem of generating the whole set of ef{\kern0pt}f{\kern0pt}icient dominating weight vectors
is open.

An extended analysis of numerical examples could show how often inef{\kern0pt}f{\kern0pt}iciency occurs and
how large dif{\kern0pt}ferences there are between an inef{\kern0pt}f{\kern0pt}icient
and an ef{\kern0pt}f{\kern0pt}icient dominating weight vector.

The ef{\kern0pt}f{\kern0pt}iciency analysis of the principal eigenvector is still incomplete.
Suf{\kern0pt}f{\kern0pt}icient conditions are discussed in
\cite{Abele-NagyBozoki2016,Abele-NagyBozokiRebak2017}: if the pairwise comparison
matrix can be made consistent by a modif{\kern0pt}ication of one or two elements (and their reciprocal),
then the eigenvector is ef{\kern0pt}f{\kern0pt}icient for (\ref{eq:xMultiObjectiveOptimizationProblem}).
It is shown in \cite{Bozoki2014} that the eigenvector
can be inef{\kern0pt}f{\kern0pt}icient even if the level of
inconsistency (as proposed by Saaty \cite{Saaty1977}, a positive
linear transformation of the maximal eigenvalue) is arbitrarily
small. However, the necessary and suf{\kern0pt}f{\kern0pt}icient condition
for the ef{\kern0pt}f{\kern0pt}iciency of the principal eigenvector is a challenging open problem.

\section*{Acknowledgements}
The authors are grateful to the three anonymous reviewers for their constructive remarks.
The comments of Michele Fedrizzi (University of Trento) are highly acknowledged.
Research was supported in part by the Hungarian Scientif{\kern0pt}ic Research Fund (OTKA) grant no.~K111797.
S.~Boz\'oki acknowledges the support of the J\'anos Bolyai Research Fellowship of the Hungarian Academy of Sciences
(no.~BO/00154/16).

\newpage
\vspace*{-11mm}

\section*{Appendix} \label{section:Appendix} 

Consider pairwise comparison matrix $\mathbf{A} \in PCM_4$ and its principal right eigenvector ${\mathbf{w}}$
in Example \ref{example:0a}, and the consistent pairwise comparison matrix
$\left[ \frac{{w}_i}{{w}_j} \right]$ as in (\ref{example:0awiperwj}).

The linear program (\ref{LP1})-(\ref{LP6}) is specif{\kern0pt}ied below.
$I=\{(2,1),(2,4),(3,1),(3,2),(4,1),(4,3)\}$ and $J=\emptyset.$

\begin{center}
\begin{tabular}{|c|c|c|c|c|c|c|c|c|c|c|c|c|}
$y_1$  & $y_2$ & $y_3$ & $y_4$  & $s_{21}$ & $s_{24}$ & $s_{31}$ & $s_{32}$ & $s_{41}$ & $s_{43}$ & & &  remark \\
\hline
   0   &    0  &   0   &   0    &    -1    &    -1     &    -1    &    -1    &      -1   &    -1     & $\rightarrow$ & $\min$  &  (\ref{LP1}) \\
\hline
\hline
1 &-1 &   0 &0& 0 &0 &0 &0 &0 &0& $\leq$ &      0     &   (\ref{LP2}), $i=2, \, j=1$  \\   
\hline
-1 &   1 &0& 0 &1& 0 &0& 0& 0& 0& $\leq$ &   0.0753   &   (\ref{LP3}), $i=2, \, j=1$  \\   
\hline
0& -1 &   0& 1 &0& 0& 0 &0& 0& 0& $\leq$ &  $-$1.6094 &   (\ref{LP2}), $i=2, \, j=4$  \\    
\hline
0& 1& 0 &-1 &   0 &1 &0& 0 &0& 0& $\leq$ &   2.1859   &   (\ref{LP3}), $i=2, \, j=4$  \\    
\hline
1& 0& -1 &   0& 0& 0 &0 &0& 0& 0& $\leq$ &   1.3863   &   (\ref{LP2}), $i=3, \, j=1$  \\    
\hline
-1 &   0 &1& 0 &0 &0& 1& 0& 0& 0& $\leq$ &  $-$1.2995 &   (\ref{LP3}), $i=3, \, j=1$  \\    
\hline
0 &1 &-1 &   0 &0& 0& 0 &0 &0& 0& $\leq$ &  1.9459    &   (\ref{LP2}), $i=3, \, j=2$  \\    
\hline
0 &-1&    1& 0 &0& 0 &0 &1 &0& 0& $\leq$ &  $-$1.3749 &   (\ref{LP3}), $i=3, \, j=2$  \\    
\hline
1 &0 &0& -1 &   0 &0 &0 &0& 0 &0& $\leq$ &  2.1972    &   (\ref{LP2}), $i=4, \, j=1$  \\    
\hline
-1&    0& 0& 1 &0 &0 &0 &0& 1 &0& $\leq$ & $-$2.1106  &   (\ref{LP3}), $i=4, \, j=1$              \\    
\hline
0 &0& 1& -1&    0& 0 &0& 0 &0 &0& $\leq$ & 1.3863     &   (\ref{LP2}), $i=4, \, j=3$  \\    
\hline
0& 0 &-1 &   1 &0& 0& 0 &0& 0& 1& $\leq$ & $-$0.8111  &   (\ref{LP3}), $i=4, \, j=3$  \\    
\hline
   0   &    0  &   0   &   0    &     1    &    0     &     0    &    0    &       0   &    0     &    $\geq$   &   0   &  (\ref{LP5}), $i=2, \, j=1$   \\
\hline
   0   &    0  &   0   &   0    &     0    &    1     &     0    &    0    &       0   &    0     &    $\geq$   &   0   &  (\ref{LP5}), $i=2, \, j=4$   \\
\hline
   0   &    0  &   0   &   0    &     0    &    0     &     1    &    0    &       0   &    0     &    $\geq$   &   0   &  (\ref{LP5}), $i=3, \, j=1$   \\
\hline
   0   &    0  &   0   &   0    &     0    &    0     &     0    &    1    &       0   &    0     &    $\geq$   &   0   &  (\ref{LP5}), $i=3, \, j=2$   \\
\hline
   0   &    0  &   0   &   0    &     0    &    0     &     0    &    0    &       1   &    0     &    $\geq$   &   0   &  (\ref{LP5}), $i=4, \, j=1$   \\
\hline
   0   &    0  &   0   &   0    &     1    &    0     &     0    &    0    &       0   &    1     &    $\geq$   &   0   &  (\ref{LP5}), $i=4, \, j=3$   \\
\hline
   1   &    0  &   0   &   0    &     0    &    0     &     0    &    0    &       0   &    0     &     =       &   0   &      (\ref{LP6})   \\
\hline
\end{tabular}
\end{center}

The f{\kern0pt}irst two inequalities belong to $i=2, j=1.$ Right hand sides are calculated as
$-b_{21}=-\log(a_{21})=-\log 1 = 0$ in (\ref{LP2}) and
$v_2-v_1 = \log(w_2)-\log(w_1) \approx 0.0753$ in (\ref{LP3}).\\  

The optimal solution is
$y^{\ast}_1 = y^{\ast}_2 = 0,
y^{\ast}_3 = -1.3749,   
y^{\ast}_4 = -2.1859,   
s^{\ast}_{21} = s^{\ast}_{31} = s^{\ast}_{41} = 0.0753,   
s^{\ast}_{24} = s^{\ast}_{32} = s^{\ast}_{43} = 0,$
and the optimum value is $-s^{\ast}_{21}-s^{\ast}_{31}-s^{\ast}_{41} = -0.226.$  
In order to make the weight vector $\mathbf{x}^{\ast} = \exp({\mathbf{y}}^{\ast}) =
(1, 1, 0.2529, 
0.1124 
)$
comparable to ${\mathbf{w}}$, renormalize it by a multiplication by
${w}_2/x^{\ast}_2 = {w}_3/x^{\ast}_3 = {w}_4/x^{\ast}_4 = 0.436173$ to get
 \[
{\mathbf{w}}^{\ast} =
\begin{pmatrix}
\mathbf{0.436173}  \\
0.436173 \\
0.110295  \\
0.049014
\end{pmatrix},
\]
that dif{\kern0pt}fers from ${\mathbf{w}}$ in its f{\kern0pt}irst coordinate only.
Furthermore, $\mathbf{w}^{\ast}$ equals to weight vector $ \mathbf{w}^{\prime\prime}$ in Example \ref{example:0b}.
One can check from the consistent approximation generated by weight vector ${\mathbf{w}}^{\ast},$
\[
\left[ \frac{{w}^{\ast}_i}{{w}^{\ast}_j} \right] =
\begin{pmatrix}
$\,$   1    $\,\,$ & $\,\,$   \mathbf{1}    $\,\,$ & $\,\,$ \mathbf{3.9546} $\,\,$ & $\,\,$ \mathbf{8.8989} $\,$    \\
$\,$ \mathbf{1}    $\,\,$ & $\,\,$   1    $\,\,$ & $\,\,$ 3.9546 $\,\,$ & $\,\,$ 8.8989 $\,$    \\
$\,$ \mathbf{0.2529} $\,\,$ & $\,\,$ 0.2529 $\,\,$ & $\,\,$   1    $\,\,$ & $\,\,$ 2.2503 $\,$   \\
$\,$ \mathbf{0.1124} $\,\,$ & $\,\,$ 0.1124 $\,\,$ & $\,\,$ 0.4444 $\,\,$ & $\,\,$   1    $\,$
\end{pmatrix} \qquad
\left( =  \left[ \frac{{w}^{\prime\prime}_i}{{w}^{\prime\prime}_j} \right]
\text{ as in (\ref{example:0bwiperwj}) } \right).
\]
that weight vector ${\mathbf{w}}^{\ast}$ internally dominates ${\mathbf{w}}$, moreover,
Theorem \ref{thm:LP} guarantees that weight vector ${\mathbf{x}}^{\ast}$ is ef{\kern0pt}f{\kern0pt}icient,
so is weight vector ${\mathbf{w}}^{\ast}$.
\end{document}